\documentclass[12pt,a4paper]{article}

\usepackage{amsmath,amsfonts,amssymb}
\usepackage{bbm}
\usepackage{cases}

\usepackage{lipsum}
\usepackage{subcaption}

\usepackage{xcolor}
\usepackage{graphicx}
\usepackage{ulem}

\newtheorem{defn}{Definition}[section]
\newtheorem{thm}[defn]{Theorem}

\newtheorem{prop}[defn]{Proposition}
\newtheorem{cor}[defn]{Corollary}
\newtheorem{rmk}[defn]{Remark}

\newenvironment{proof}{{\bf Proof }}{{\vskip 0.1cm \hfill$\Box$}}

\begin{document} 

\noindent
{\Large \bf Quantitative analysis for $L^2$-estimates in linear elliptic equations via divergence-free transformation}
\\ \\
\bigskip
\noindent
{\bf Haesung Lee}  \\
\noindent
{\bf Abstract.}  
This paper establishes an explicit $L^2$-estimate for weak solutions $u$ to linear elliptic equations in divergence form with general coefficients and external source term $f$, stating that the $L^2$-norm of $u$ over $U$ is bounded by a constant multiple of the $L^2$-norm of $f$ over $U$. In contrast to classical approaches based on compactness arguments, the proposed method, which employs a divergence-free transformation method, provides a computable and explicit constant $C>0$. The $L^2$-estimate remains robust even when there is no zero-order term, and the analysis further demonstrates that the constant $C>0$ decreases as the diffusion coefficient or the zero-order term increases. These quantitative results provide a rigorous foundation for applications such as a posteriori error estimates in Physics-Informed Neural Networks (PINNs), where explicit error bounds are essential.


\section{Introduction} \label{intro}
This paper deals with the applied aspects of the recent work \cite{L25jm}, focusing in particular on the quantitative analysis of the constant $C>0$ in the following $L^2$-estimate
\begin{equation} \label{l2estimate}
\|u \|_{L^2(U)} \leq C \|f \|_{L^2(U)}
\end{equation}
for the unique weak solution to the homogeneous boundary value problem for elliptic equations in divergence form (see Definition \ref{soltoweak}):
\begin{equation} \label{maineq2}
\left\{
\begin{alignedat}{2}
-{\rm div}(\gamma A \nabla u) + \langle \mathbf{H}, \nabla u \rangle + (c+\alpha) u &=f  && \quad \; \text{in } U,\\
u &= 0 &&\quad \; \text{on } \partial U,
\end{alignedat} \right.
\end{equation}
where the specific assumptions on the coefficients are stated in Theorem \ref{maintheor}.\\
From the perspective of pure mathematics in the theory of partial differential equations, one might be satisfied with simply obtaining the estimate \eqref{l2estimate}, proving the existence of a constant $C>0$, and demonstrating its independence from certain quantities such as $f$. However, in applied mathematics, it is often necessary to know this constant explicitly, and understanding its specific dependencies on other quantities is of critical importance.
For instance, in the case of Physics-Informed Neural Networks (PINNs), a prominent recent PDE solver technique based on artificial intelligence, one considers a trial function $\Psi \in H^{1,2}_0(U) \cap C^{\infty}(\mathbb{R}^d)$ (see \cite{LLF98, BN18, MZ22}).  By imposing appropriate conditions on the coefficients in \eqref{maineq2}, the existence and uniqueness of the solution, together with the corresponding $L^2$-estimate \eqref{l2estimate}, immediately yields the following a posteriori error estimate (cf.~Corollary \ref{pineroestim}):
\begin{equation} \label{pinnesti}
\Vert u-\Psi \Vert_{L^2(U)} \leq C \Vert f- \mathcal{L} \Psi \Vert_{L^2(U)},
\end{equation}
where 
\begin{equation} \label{lpsidefined}
\mathcal{L}\Psi = -{\rm div}(\gamma A \nabla \Psi) + \langle \mathbf{H}, \nabla \Psi \rangle + (c+\alpha) \Psi,
\end{equation}
(cf. Proposition \ref{convdivnon}).
Ultimately, by squaring both sides of \eqref{pinnesti} and reformulating it for Monte Carlo integration, we observe that as the right-hand side of the inequality (representing the residual that defines part of the loss function) decreases, the error on the left-hand side also diminishes.  In this context, a quantitative analysis of the constant $C>0$ is essential for determining the rate at which the error diminishes. \\
Indeed, the estimate \eqref{l2estimate} can be derived with an explicit computation of the constant $C>0$, utilizing the coercivity of the coefficients. When ${\rm div}\,\mathbf{H}$ is negative or $\alpha$ is sufficiently large, appropriate energy estimates, together with the Sobolev embedding theorem, may yield the desired result \eqref{l2estimate} (cf.~\cite[Lemma~2.1.3]{BKRS15}, \cite[Theorem~1.1]{L24}).
However, for a general diffusion operator, such coercivity conditions may not hold.  For instance, ${\rm div} \mathbf{H}$ could be a large positive quantity or a general drift $\mathbf{H}$ is given, or $c+\alpha$ might even be identically zero.  In such scenarios, deriving \eqref{l2estimate} and explicitly computing the constant $C>0$ via coercivity conditions becomes difficult.\\
A notable alternative for establishing \eqref{l2estimate} is the compactness method through a proof by contradiction. For instance, in \cite[Section 6.2, Theorem 6]{E10}(cf. \cite[Lemma 9.17]{GT}), assuming that \eqref{l2estimate} fails, one ultimately derives a contradiction, thus proving the validity of the estimate \eqref{l2estimate}. The key tools employed in this method are the Rellich–Kondrachov theorem (\cite[Section 5.7, Theorem 1]{E10})
and weak compactness arguments, along with a condition ensuring the uniqueness of the solution, which is essential for the argument to hold. For instance, for an operator $\mathcal{L}$ satisfying the weak maximum principle (cf. \cite[Theorem 2.1.8]{BKRS15}), an $L^2$-estimate like \eqref{l2estimate} can be established.   Another approach to obtain the estimate \eqref{l2estimate} is to first establish the existence and uniqueness of the solution, and then directly deduce the existence of a constant $C > 0$ by applying the bounded inverse theorem from functional analysis (see \cite[Corollary 2.7]{Br11}).
However, the most significant drawback of these approaches is that, although we can ascertain that the constant $C>0$ is independent of $f$, we obtain no information regarding its dependence on other quantities, such as the coefficients or the domain in \eqref{maineq2}. Therefore, since this method yields a non-constructive constant $C>0$, it is difficult to apply it to the error analysis of PINNs, where an explicit form of the constant is essential to derive quantitative estimates.   \\
As one possible approach to obtaining a concrete constant $C>0$, we refer to the following contraction estimate recently derived in \cite[Theorem~1.1(iii)]{L25jm} (cf. \cite[Theorem~1.1(ii)]{L24}):
\begin{equation} \label{l2contraesticl}
\|u \|_{L^2(U)} \leq \frac{K_1}{\alpha} \|f\|_{L^2(U)},
\end{equation}
where $K_1 \geq 1$ is the constant appearing in Theorem \ref{existrho}(ii).
Indeed, while the term $\frac{K_1}{\alpha}$ in \eqref{l2contraesticl} decreases as $\alpha$ increases, it may become large and thus inefficient for very small $\alpha$. Our main result, Theorem \ref{maintheor}, not only ensures the $L^2$-estimate remains robust even for such small values of $\alpha$, but also provides an improvement over the constant $\frac{K_1}{\alpha}$ in \eqref{l2contraesticl} (see Theorem \ref{maintheor}(iii) and \eqref{constatcompar} in Section \ref{conclrema}).\\
Before stating our main result, we introduce the primary condition on $A$ and $\mathbf{H}$ on $U$. \\

\noindent
{\bf (S)}:
{\it
$U$ is a bounded open subset of $\mathbb{R}^d$ with $d \geq 2$, $B_r(x_0)$ is an open ball in $\mathbb{R}^d$ with $\overline{U} \subset B_r(x_0)$, 
$\mathbf{H} \in L^p(U, \mathbb{R}^d)$, $h \in L^p(U)$  
with $p \in (d, \infty)$ satisfies $\| \mathbf{H}\| \leq h$ in $U$, and  $A=(a_{ij})_{1 \leq i,j \leq d}$ is a (possibly non-symmetric) matrix of measurable functions on $\mathbb{R}^d$ such that for some constants $M>0$ and $\lambda>0$ it holds that for a.e. $x \in \mathbb{R}^d$ and for all $\xi \in \mathbb{R}^d$,
\begin{equation*} 
\max_{1 \leq i,j \leq d} |a_{ij}(x)| \leq M, \;\; \; \; \langle A(x) \xi, \xi \rangle \geq \lambda \| \xi \|^2. \;\; 
\end{equation*}
}

\begin{thm} \label{maintheor}
Assume that {\bf (S)} holds. Let $\gamma \in [1, \infty)$ and $\alpha \in [0, \infty)$ be constants.  Let $\hat{d}:=d$ if  $d \geq 3$ and $\hat{d}$ is an arbitrary number in $(2, \infty)$ if $d=2$. Assume that $c \in L^1(U)$ with $c \geq 0$.  
Additionally, suppose that either condition {\rm (a)} or {\rm (b)} holds:
\begin{itemize}
\item[\rm (a)] $f \in L^q(U)$ with $q > \frac{d}{2}$ and $d \geq 2$
\item[\rm (b)] $f \in L^{\frac{2d}{d+2}}(U)$ and $c \in L^{\frac{2d}{d+2}}(U)$ with $d \geq 3$.
\end{itemize}
Then, there exists a unique weak solution $u \in H^{1,2}_0(U)$ to \eqref{maineq2}. In particular, $u \in H^{1,2}_0(U) \cap L^{\infty}(U)$ if (a) is assumed.
Moreover, under the assumption that either (a) or (b) holds, $u$ satisfies the following estimates:
\begin{itemize}
\item[(i)]
If $f \in L^{\frac{2\hat{d}}{\hat{d}+2}}(U)$, then
\begin{equation} \label{energyesma}
\|\nabla u\|_{L^{2}(U)} \leq \frac{K_1 \hat{k}}{\lambda \gamma} \|f \|_{L^{\frac{2 \hat{d}}{\hat{d}+2}}(U)}
\end{equation}
and
\begin{equation} \label{newl2estimviaen}
\| u \|_{L^2(U)} \leq \left(  \frac{d^2 \gamma \lambda}{8(d-1)^2 |U|^{\frac2d}} + \alpha \right)^{-\frac12} \left(  \frac{K_1^2 \hat{k}^2}{2 \gamma \lambda}	\right)^{\frac12} \|f \|_{L^{\frac{2 \hat{d}}{\hat{d}+2}}(U)}.
\end{equation}
where $K_1 \geq 1$ is a constant as in Theorem \ref{existrho} and $\hat{k}:= \frac{\hat{d}}{\hat{d}-2} |U|^{\frac12 -\frac{1}{\hat{d}}}$ if $d=2$ and $\hat{k}:=\frac{2(d-1)}{d-2}$ if $d \geq 3$.
\item[(ii)]
If $f \in L^2(U)$, then
\begin{equation} \label{originl2estima}
\|u\|_{L^2(U)} \leq K_1\left( \frac{d^2 \gamma \lambda}{4(d-1)^2 |U|^{\frac2d}} + \alpha \right)^{-1} \|f\|_{L^2(U)}
\end{equation}
and
\begin{equation} \label{contraestiml2}
\|u\|_{L^2(U)} \leq K_1^{\frac12}\left( \frac{d^2 \gamma \lambda}{4K_1(d-1)^2 |U|^{\frac{2}{d}}}  +\alpha  \right)^{-1} \|f\|_{L^2(U)}.
\end{equation}
\item[(iii)]
Let $C:=\min\Bigg( \left(  \frac{d^2 \gamma \lambda}{8(d-1)^2 |U|^{\frac2d}} + \alpha \right)^{-\frac12} \left(  \frac{K_1^2 \hat{k}^2}{2 \gamma \lambda}	\right)^{\frac12} |U|^{\frac{1}{\hat{d}}}, \;\;K_1\left( \frac{d^2 \gamma \lambda}{4(d-1)^2 |U|^{\frac2d}} + \alpha \right)^{-1}, \\
\;\; K_1^{\frac12}\left( \frac{d^2 \gamma \lambda}{4K_1(d-1)^2 |U|^{\frac{2}{d}}}  +\alpha  \right)^{-1}	    \Bigg)$, 
where $\hat{k}$ is defined as in (i). If $f \in L^2(U)$, then \eqref{l2estimate} holds.
\end{itemize}
\end{thm}
As a direct consequence of Theorem~\ref{maintheor}, we obtain the following a posteriori error estimate using a trial function $\Psi \in H^{1,2}_0(U) \cap C^2(\mathbb{R}^d)$.
\begin{cor}\label{pineroestim}
Assume {\bf (S)} and that ${\rm div}A \in L^2(U, \mathbb{R}^d)$. Let $\gamma \in [1, \infty)$ and $\alpha \in [0, \infty)$ be constants. Suppose that $c \in L^2(U)$ with $c \geq 0$, and that $f \in L^2(U)$. Let $\Psi \in H^{1,2}_0(U) \cap C^2(\mathbb{R}^d)$.
Then, \eqref{pinnesti} holds, where $C>0$ is the constant from Theorem \ref{maintheor}(iii) and $\mathcal{L}\Psi$ is defined as in \eqref{lpsidefined}.
\end{cor}
A key ingredient in the proof of Theorem~\ref{maintheor} is the divergence-free transformation developed in \cite{L25jm} (see Theorem~\ref{mainequivthm}).
This transformation effectively absorbs the drift term $\mathbf{H}$ into a modified diffusion matrix and yields an equivalent equation with a divergence-free drift $\mathbf{B}$, which plays a crucial role in deriving our estimates.\\
This paper is organized as follows.
Section~\ref{framew} introduces the essential notations and key ingredients required for the proof of our main results.
Section~\ref{profofmainre} presents the proofs of the main results.
Finally, Section~\ref{conclrema} provides concluding remarks in this paper with further discussion.

\section{Framework} \label{framew}
In this paper, $\mathbb{R}^d$ denotes the $d$-dimensional Euclidean space with the standard inner product $\langle \cdot, \cdot \rangle$ and the corresponding Euclidean norm $\| \cdot \|$. For $x_0 \in \mathbb{R}^d$ and $r>0$, $B_r(x_0)$ denotes the open ball centered at $x_0$ with radius $r$. The Lebesgue measure on $\mathbb{R}^d$ is denoted by $dx$, and for a measurable set $E \subset \mathbb{R}^d$, $dx(E)$ is written as $|E|$. Throughout, $U$ is a bounded open subset of $\mathbb{R}^d$. Let $s \in [1, \infty]$. $L^s(U)$ denotes the standard Lebesgue space of functions on $U$ that are $s$-integrable with respect to the Lebesgue measure $dx$, equipped with the norm $\|\cdot\|_{L^s(U)}$. Similarly, $L^s(U, \mathbb{R}^d)$ is the space of vector fields whose  components belong to $L^s(U)$, equipped with the norm, $\| \mathbf{F} \|_{L^s(U, \mathbb{R}^d)}:=\big \| \| \mathbf{F} \| \big \|_{L^s(U)}$, $\mathbf{F} \in L^s(U, \mathbb{R}^d)$. $C(U)$ and $C(\overline{U})$ are the spaces of continuous functions on $U$ and its closure $\overline{U}$, respectively. For $k \in \mathbb{N} \cup \{\infty\}$, $C^k(U)$ is the space of $k$-times continuously differentiable functions, and $C_0^k(U)$ denotes the subset of functions in $C^k(U)$ with compact support in $U$. The Sobolev space $H^{1,2}(U)$ consists of functions $f \in L^2(U)$ whose weak partial derivatives, denoted by $\partial_i f$, also belong to $L^2(U)$. The weak gradient is written as $\nabla f = (\partial_1 f, \ldots, \partial_d f)$. The space $H^{1,2}_0(U)$ is defined as the closure of $C_0^{\infty}(U)$ in the $H^{1,2}(U)$-norm. For convenience, we also define the set of bounded functions in $H^{1,2}_0(U)$ as $H^{1,2}_0(U)_b := H^{1,2}_0(U) \cap L^{\infty}(U)$. The Sobolev space $H^{2,s}(U)$ is the space of functions $f \in L^s(U)$ whose first and second weak derivatives $\partial_i f$ and $\partial_i \partial_j f$ belong to $L^s(U)$ for all $i, j = 1, \ldots, d$, equipped with the standard $H^{2,s}(U)$-norm.
The weak Hessian matrix of a function $f$ is denoted by $\nabla^2 f := (\partial_i \partial_j f){1 \leq i,j \leq d}$. The trace of a square matrix $B = (b_{ij})_{1 \leq i,j \leq d}$ is defined by $\operatorname{trace}(B) := \sum_{i=1}^d b_{ii}$. The central object of study in this paper is the weak solution to the elliptic boundary value problem, defined as follows.

\begin{defn} \label{soltoweak}
Let $\gamma \in [1, \infty)$ and $\alpha \in [0, \infty)$. Let $A = (a_{ij})_{1 \leq i,j \leq d}$ be a matrix of bounded and measurable functions on $\mathbb{R}^d$, and let $\mathbf{H} \in L^2(U, \mathbb{R}^d)$, $c \in L^1(U)$ and $f \in L^1(U)$.  
We say that $u$ is a weak solution to \eqref{maineq2} if $u \in H^{1,2}_0(U)$ with $(c+\alpha)u \in L^1(U)$ satisfies
$$
\int_U \langle \gamma A \nabla u, \nabla \psi \rangle + \langle \mathbf{H}, \nabla u \rangle \psi + (c+\alpha) u \psi \, dx = \int_U f \psi \, dx \quad \text{for all } \psi \in C_0^{\infty}(U).
$$
\end{defn}
The following provides the definitions for the divergence of matrix valued functions and vector fields. This concept is central to Corollary 1.2, which requires the assumption that ${\rm div}\,A \in L^2(U, \mathbb{R}^d)$ to establish the a posteriori error estimate for Physics-Informed Neural Networks (PINNs). The divergence of the matrix $A=(a_{ij})_{1 \leq i,j \leq d}$ is a vector field whose $j$-th component is the weak divergence of the $j$-th column of $A$.
\begin{defn}
\label{basidefn}
\begin{itemize}
\item[(i)] Let $U \subset \mathbb{R}^d$ be an open set with $d \geq 2$, $B = (b_{ij})_{1 \leq i,j \leq d}$ be a (possibly non-symmetric) matrix of locally integrable functions on $U$, and $\mathbf{E}=(e_1,\ldots, e_d) \in L^1_{loc}(U, \mathbb{R}^d)$. We write ${\rm div} B = \mathbf{E}$ if
\begin{equation} \label{defndivma}
\int_{U} \sum_{i,j=1}^d b_{ij} \partial_i \phi_j \, dx = - \int_{U} \sum_{j=1}^d e_j \phi_j \, dx \quad \text{ for all $\phi_j \in C_0^{\infty}(U)$, $j=1, \ldots, d$}.
\end{equation}
In other words, if we consider the family of column vectors
\[
\mathbf{b}_j=\begin{pmatrix}
b_{1j} \\ \vdots \\ b_{dj}
\end{pmatrix},\quad 1\le j\le d,\  \text{ i.e.}\  B=(\mathbf{b}_1 | \ldots |\mathbf{b}_d),
\]
then \eqref{defndivma} can be rewritten using inner products as
\begin{equation*}
\int_{U} \sum_{j=1}^d \langle \mathbf{b}_j, \nabla \phi_j \rangle \, dx = - \int_{U} \langle \mathbf{E}, \boldsymbol{\Phi} \rangle \, dx,
\end{equation*}
for all test vector fields $\boldsymbol{\Phi}=(\phi_1, \ldots, \phi_d) \in (C_0^{\infty}(U))^d$.\
\item[(ii)] Let $U \subset \mathbb{R}^d$ be an open set with $d \geq 2$, $\mathbf{F} \in L^1_{loc}(U, \mathbb{R}^d)$ and $f \in L^1_{loc}(U)$. We write ${\rm div} \mathbf{F}=f$ if
$$
\int_{U} \langle \mathbf{F}, \nabla \varphi \rangle \, dx = -\int_{U} f \varphi \, dx, \quad \text{ for all } \varphi \in C_0^{\infty}(U).
$$
\end{itemize}
\end{defn}
The following lemma, proven in \cite{L25bv}, guarantees that $A \nabla u$ has a weak divergence.
\begin{prop}{\bf (\cite[Proposition 5.6]{L25bv})} \label{convdivnon}    \\
Let $d \geq 2$, $U$ be an open subset of $\mathbb{R}^d$, $u \in H^{2,1}_{loc}(U)$ with $\nabla u \in L^2_{loc}(U, \mathbb{R}^d)$, and $A=(a_{ij})_{1 \leq i,j \leq d}$ be a (possibly non-symmetric)
matrix of functions in $L^{\infty}_{loc}(U)$ with ${\rm div} A \in L_{loc}^2(U, \mathbb{R}^d)$. Then, 
$$
{\rm div } (A \nabla u) = {\rm trace}(A \nabla^2 u) + \langle {\rm  div} A, \nabla u \rangle \;\; \text{ in $U$}.
$$
\end{prop}

The following is a basic approximation result from \cite{L25jm}, which is used in the proof of Theorem~\ref{maintheor}.
\begin{prop} {\bf (\cite[Proposition A.8]{L25jm})} \label{propapprobdd} \\
Let $v \in H^{1,2}_0(U)_b$. Then, there exist a constant $M>0$ which only depends on $\|v\|_{L^{\infty}(U)}$ and a sequence of functions $(v_n)_{n \geq 1}$ in $C_0^{\infty}(U)$ such that $\sup_{n \geq 1} \|v_n \| \leq M$,
 $\lim_{n \rightarrow \infty} v_n = v$ a.e. on $U$ and $\lim_{n \rightarrow \infty} v_n = v$ in $H^{1,2}_0(U)$.
\end{prop}
From now on, we adopt the main approach developed in \cite{L25jm}.  
The central idea of \cite{L25jm} is to transform \eqref{maineq2} into an equation with divergence-free drifts (referred to as the divergence-free transformation).  
To this end, the key step is to construct an appropriate weight $\rho\,dx$ as in the following result. Remarkably, the constant $K_1 \geq 1$ in the following result is independent of $\gamma \in [1, \infty)$. The fundamental idea behind the following result originates from \cite[Theorem 1]{BRS12} (cf. \cite[Chapter 2]{BKRS15}).
\begin{thm}{\bf (\cite[Theorem 3.1]{L25jm})} 
\label{existrho} \\[4pt]
Assume that {\bf (S)} holds and let $\gamma \in [1, \infty)$. Then, the following hold:
\begin{itemize}
\item[(i)]
Let $x_1 \in U$. Then, there exists $\rho \in H^{1,2}(B_{4r}(x_0)) \cap C(B_{4r}(x_0))$ with $\rho(x)>0$ for all $x \in B_{4r}(x_0)$ and $\rho(x_1)=1$ such that
\begin{equation} \label{infinva}
\int_{B_{4r}(x_0)} \langle \gamma A^T \nabla \rho  + \rho \mathbf{H}, \nabla \varphi \rangle dx = 0, \quad \text{ for all $\varphi \in C_0^{\infty}(B_{4r}(x_0))$}.
\end{equation}
\item[(ii)]
Let $\rho$ be as in Theorem \ref{existrho}(i). Then, there exists a constant $K_1 \geq 1$ which only depends on $d$, $\lambda$, $M$, $r$, $p$, $\|h\|_{L^p(U)}$ such that \;$\max_{\overline{B}_{3r}(x_0)} \rho  \leq K_1 \min_{\overline{B}_{3r}(x_0)} \rho$, and hence
$$
1 \leq \max_{\overline{U}} \rho  \leq K_1 \min_{\overline{U}} \rho \leq K_1.
$$
\end{itemize}
\end{thm}
\begin{proof}
(i) By \cite[Theorem~3.1(i)]{L25jm}, we have $\rho \in H^{1,2}(B_{4r}(x_0)) \cap C(B_{4r}(x_0))$ with $\rho(x) > 0$ for all $x \in B_{4r}(x_0)$ and $\rho(x_1) = 1$, such that
\begin{equation} \label{scalgamequ}
\int_{B_{4r}(x_0)} \left\langle A^T \nabla \rho + \rho \left(\frac{1}{\gamma} \mathbf{H} \right), \nabla \varphi \right\rangle \, dx = 0, \quad \text{for all } \varphi \in C_0^{\infty}(B_{4r}(x_0)).
\end{equation}
Thus, \eqref{infinva} holds, and hence the assertion (i) follows.\\
(ii) Applying \cite[Theorem~3.1(ii)]{L25jm} to \eqref{scalgamequ} and noting that $\| \frac{1}{\gamma} \mathbf{H} \| \leq h$ in $U$, we obtain the desired result for (ii).
\end{proof}
The following is the main ingredient in this paper that establishes the equivalence between \eqref{maineq2} and an equation with divergence-free drifts.
\begin{thm}{\bf (\cite[Theorem 3.2]{L25jm}, Divergence-free transformation)} \label{mainequivthm}  \\
Assume that {\bf (S)} holds.
Let $\rho \in H^{1,2}(U) \cap C(\overline{U})$ be a strictly positive function on $\overline{U}$ constructed as in Theorem \ref{existrho}.
Define 
\begin{equation} \label{divfreevect}
\mathbf{B}:=\mathbf{H}+\frac{1}{\rho} (\gamma A)^T \nabla \rho \quad \text{on $U$}.
\end{equation}
Then, $\rho \mathbf{B} \in L^2(U, \mathbb{R}^d)$ and 
\begin{equation} \label{divfreecondi}
\int_{U} \langle \rho \mathbf{B}, \nabla \varphi \rangle dx = 0, \quad \text{ for all $\varphi \in C_0^{\infty}(U)$}.
\end{equation}
Let $\gamma \in [1, \infty)$ and $\alpha \in [0, \infty)$ be constants. Let $f \in L^1(U)$ and $u \in H^{1,2}_0(U)$ with $(c+\alpha)u \in L^1(U)$. Then the following (i) and (ii) are equivalent.
\begin{itemize}
\item[(i)]
\begin{equation} \label{maineqfir}
\int_U \langle \gamma A \nabla u, \nabla \psi \rangle +\langle \mathbf{H}, \nabla u \rangle \psi  + (c+\alpha)u \psi \,dx = \int_U f \psi\, dx \quad \text{ for all } \psi \in C_0^{\infty}(U).
\end{equation}
\item[(ii)]
\begin{equation} \label{maineqfdiv}
\int_U \langle \rho (\gamma A) \nabla u, \nabla \varphi \rangle   +  \langle \rho \mathbf{B}, \nabla u \rangle \varphi   + \rho (c+\alpha) u \varphi \,dx = \int_U f \rho \varphi  \,dx \quad \text{ for all } \varphi \in C_0^{\infty}(U).
\end{equation}
\end{itemize}
\end{thm}
\begin{proof}
The assertion immediately follows from \cite[Theorem~3.2]{L25jm}, where $A$ and $c$ in \cite[Theorem~3.2]{L25jm} are replaced by $\gamma A$ and $c+\alpha$ in the present setting, respectively.
\end{proof}
The following result provides an explicit computation of the constant in the Poincaré type inequality derived from the Gagliardo–Nirenberg–Sobolev inequality
\cite[Section 5.6, Theorem 1]{E10} (cf. \cite[Theorem 4.8]{EG15}).
\begin{prop}  \label{mainpoincare}
The following inequality holds:
\begin{align} \label{poincareineq}
\|f\|_{L^2(U)} \leq \frac{2(d-1)}{d} |U|^{\frac{1}{d}} \| \nabla f\|_{L^2(U)} \quad \text{for all $f \in H^{1,2}_0(U)$.}
\end{align}
\end{prop}
\begin{proof}
By applying the Gagliardo--Nirenberg--Sobolev inequality \cite[Section~5.6, Theorem~1]{E10} and H\"{o}lder's inequality, we obtain
\begin{align*}
\|f\|_{L^2(U)}& \leq \frac{\frac{2d}{d+2} (d-1)}{d-\frac{2d}{d+2}} \| \nabla f \|_{L^{\frac{2d}{d+2}}(U)} \\
& \leq \frac{2(d-1)}{d} |U|^{\frac{1}{d}} \| \nabla f\|_{L^2(U)} \quad \text{for all $f \in H^{1,2}_0(U)$.} 
\end{align*}
\end{proof}
\begin{rmk}
Let $d \geq 3$. From \cite[Section~5.6, Theorem~1]{E10} and H\"{o}lder's inequality, one can also derive the following inequalities:
\begin{equation} \label{remsobol}
\|f\|_{L^2(U)} \leq |U|^{\frac{1}{d}} \|f\|_{L^{\frac{2d}{d-2}}(U)} \leq \frac{2(d-1)}{d-2} |U|^{\frac{1}{d}}  \| \nabla f\|_{L^2(U)} \quad \text{for all $f \in H^{1,2}_0(U)$.}
\end{equation}
However, in comparison with \eqref{poincareineq}, since $\frac{2(d-1)}{d} \leq \frac{2(d-1)}{d-2}$, the inequality \eqref{poincareineq} provides a better constant than \eqref{remsobol}.
\end{rmk}
\section{Proof of main results} \label{profofmainre}
We now prove our main results by using the results from Section~\ref{framew} as the main ingredient.  
In fact, under the assumption (a) in Theorem~\ref{maintheor}, the existence, uniqueness, and boundedness of solutions were established in \cite[Theorem~1.1]{L25jm}. However, the main contribution of the present paper lies in addressing assumption (b).  
In particular, under (b), we require only the condition $f \in L^{\frac{2d}{d+2}}(U)$, and we explicitly present the constants arising from the energy  and $L^2$-estimates, so that we can explicitly compute the constant $C>0$ in \eqref{l2estimate}.
\\
\centerline{} 
{\bf Proof of Theorem \ref{maintheor}}\\
{\sf \underline{Step 1}}: Before proving the estimates in (i)--(iii), we first establish the existence and uniqueness of the weak solution to \eqref{maineq2}.  
First, the uniqueness immediately follows from \cite[Theorem~1.1(ii)]{L25jm}.  
Under the assumption (a) in Theorem~\ref{maintheor}, the existence of a weak solution $u \in H^{1,2}_0(U) \cap L^{\infty}(U)$ to \eqref{maineq2} follows from \cite[Theorem~1.1(i)]{L25jm}. Now, assume (b) in Theorem~\ref{maintheor}. Let $(f_n)_{n \geq 1} \subset L^{\infty}(U)$ be a sequence such that $\lim_{n \rightarrow \infty} f_n = f$ a.e. in $U$ and $|f_n| \leq |f|$ in $U$ for all $n \geq 1$. 
By \cite[Theorem~1.1(ii)]{L25jm}, for each $n \geq 1$, there exists $u_n \in H^{1,2}_0(U)  \cap L^{\infty}(U)$ satisfying
\begin{equation} \label{approequat}
\int_U \langle \gamma A \nabla u_n, \nabla \psi \rangle + \langle \mathbf{H}, \nabla u_n \rangle \psi + (c+\alpha) u_n \psi \, dx = \int_U f_n \psi \, dx \quad \text{for all } \psi \in C_0^{\infty}(U),
\end{equation}
and
$$
\|u_n\|_{H^{1,2}_0(U)} \leq K_5 \|f_n\|_{L^{\frac{2d}{d+2}}(U)} \leq K_5 \|f\|_{L^{\frac{2d}{d+2}}(U)},
$$
where $K_5 > 0$ is a constant independent of $n \geq 1$.  
Using the weak compactness and Sobolev's inequality, there exist $u \in H^{1,2}_0(U)$ and a subsequence of $(u_n)_{n \geq 1} \subset H^{1,2}_0(U)  \cap L^{\infty}(U)$, still denoted by $(u_n)_{n \geq 1}$, such that
\begin{equation} \label{weaklimiun}
\lim_{n \rightarrow \infty} u_n = u \quad \text{ weakly in $H^{1,2}_0(U)$} \quad \text{ and } \quad \lim_{n \rightarrow \infty} u_n = u \text{ weakly in $L^{\frac{2d}{d-2}}(U)$}.
\end{equation}
Hence, by passing to the limit in \eqref{approequat}, we conclude that $u$ is a weak solution to \eqref{maineq2}.\\ \\
{\sf \underline{Step 2}}: Under the assumption (a), we will prove the estimates (i), (ii) and (iii).  
Let $\rho \in H^{1,2}(U) \cap C(\overline{U})$ be a strictly positive function on $\overline{U}$ constructed as in Theorem~\ref{existrho}, and define the vector field $\mathbf{B}$ as in \eqref{divfreevect}.  
First, assume that condition (a) in Theorem~\ref{maintheor} holds. Let $u \in H^{1,2}_0(U) \cap L^{\infty}(U)$ be the unique weak solution to \eqref{maineq2}.  
Since \eqref{maineqfir} holds, it follows by Theorem~\ref{mainequivthm} that equation~\eqref{maineqfdiv} is satisfied.  
By the approximation argument in Proposition \ref{propapprobdd}, \eqref{maineqfdiv} holds for all $\varphi \in H^{1,2}_0(U) \cap L^{\infty}(U)$.  
Substituting $u$ for $\varphi$, we obtain  
\begin{equation} \label{mainfdivrepu}
\int_U \langle \rho (\gamma A) \nabla u, \nabla u \rangle \, dx + \int_U \langle \rho \mathbf{B}, \nabla u \rangle u \, dx + \int_U \rho (c+\alpha) u^2 \, dx = \int_U f u \rho \, dx.
\end{equation}  
Since $\frac{1}{2} u^2 \in H^{1,2}_0(U)$ and $\nabla \left(\frac{1}{2} u^2\right) = u \nabla u$ in $U$, and since \eqref{divfreecondi} holds for all $\varphi \in H^{1,2}_0(U)$, we have  
\begin{equation} \label{diverfreepropert}
\int_{U} \langle \rho \mathbf{B}, \nabla u \rangle u \, dx =\int_{U} \left \langle \rho \mathbf{B}, \nabla \left(\frac12 u^2 \right) \right \rangle\,dx = 0.
\end{equation}
Thus, from \eqref{mainfdivrepu} and H\"{o}lder's inequality, it follows that
\begin{equation} \label{energyorign}
\left( \min_{\overline{U}} \rho \right) \left(\gamma \lambda  \| \nabla u \|^2_{L^2(U)} + \alpha \|u\|^2_{L^2(U)} \right) \leq \left( \max_{\overline{U}} \rho \right) \|f \|_{L^{\frac{2 \hat{d}}{\hat{d}+2}}(U)} \|u \|_{L^{\frac{2 \hat{d}}{\hat{d}-2}}(U)}.
\end{equation}
Particularly, observe that if $d=2$, then
\begin{equation} \label{2dimpoinc}
\|u \|_{L^{\frac{2 \hat{d}}{\hat{d}-2}}(U)} \leq \frac{\frac{\hat{d}}{\hat{d}-1}}{2 - \frac{\hat{d}}{\hat{d}-1}} \| \nabla u \|_{L^{\frac{\hat{d}}{\hat{d}-1}}(U)} \leq \frac{\hat{d}}{\hat{d}-2} |U|^{\frac{1}{2} - \frac{1}{\hat{d}}} \| \nabla u \|_{L^2(U)}.
\end{equation}
Therefore, Theorem~\ref{existrho}, \eqref{energyorign}, \eqref{2dimpoinc}, and \eqref{remsobol} imply
\begin{equation} \label{energybasi}
\gamma \lambda \| \nabla u \|^2_{L^2(U)} + \alpha \|u\|^2_{L^2(U)} \leq K_1 \hat{k} \|f \|_{L^{\frac{2 \hat{d}}{\hat{d}+2}}(U)} \| \nabla u \|_{L^2(U)}.
\end{equation}
Applying Young's inequality to \eqref{energybasi}, we have
\begin{equation} \label{enerl2estima}
\frac{\gamma \lambda}{2} \| \nabla u \|^2_{L^2(U)} + \alpha \|u\|^2_{L^2(U)} \leq \frac{K_1^2 \hat{k}^2}{2 \gamma \lambda} \|f \|^2_{L^{\frac{2 \hat{d}}{\hat{d}+2}}(U)}.
\end{equation}
Hence, \eqref{energyesma} follows. Applying Proposition~\ref{mainpoincare} to \eqref{enerl2estima}, \eqref{newl2estimviaen} is obtained.\\
Meanwhile, it follows from \eqref{mainfdivrepu}, \eqref{diverfreepropert}, Proposition~\ref{mainpoincare} and H\"{o}lder's inequality that
\begin{equation} \label{mainl2estimre}
\left( \min_{\overline{U}} \rho \right) \left( \frac{d^2 \gamma \lambda}{4(d-1)^2 |U|^{\frac{2}{d}}} + \alpha \right) \|u \|^2_{L^2(U)} \leq \left( \max_{\overline{U}} \rho \right) \|f \|_{L^2(U)} \|u \|_{L^2(U)}.
\end{equation}
Applying Theorem~\ref{existrho} to \eqref{mainl2estimre}, \eqref{originl2estima} follows. \\
On the other hand, substituting $u$ for $\varphi$ in \eqref{mainfdivrepu} and using \eqref{diverfreepropert}, Proposition~\ref{mainpoincare} and Theorem \ref{existrho}, we have
$$
\left( \frac{d^2 \gamma \lambda}{4K_1(d-1)^2 |U|^{\frac{2}{d}}}  +\alpha  \right)\|u\|_{L^2(U, \rho dx)}^2  \leq \|f\|_{L^2(U, \rho dx)} \|u \|_{L^2(U, \rho dx)},
$$
and hence \eqref{contraestiml2} follows from Theorem \ref{existrho}. Finally, (iii) holds by (i) and (ii). \\ \\
{\sf \underline{Step 3}}: Under the assumption (b), we will prove the estimates (i), (ii) and (iii).  Indeed, based on the approximations in {\sf Step 1} and using the results in {\sf Step 2}, we obtain the estimates (i)--(iii) where $u$ and $f$ are replaced by $u_n$ and $f_n$, respectively. Using the fact that $|f_n| \leq |f|$ in $U$ and the limit in \eqref{weaklimiun}, the assertion follows. \\
\centerline{}
\centerline{}
{\bf Proof of Corollary \ref{pineroestim}}\\
Since ${\rm div} A \in L^2(U, \mathbb{R}^d)$, using integration by parts in Proposition \ref{convdivnon}, $\mathcal{L} \Psi \in L^2(U)$ and $\Psi$ is a weak solution to \eqref{maineq2} where $f$ is replaced by $\mathcal{L} \Psi$. By linearity, $u-\Psi$ is a weak solution to \eqref{maineq2} where $f$ is replaced by $f-\mathcal{L} \Psi$. Applying Theorem \ref{maintheor} to $u-\Psi$, the assertion follows.
\section{Discussion} \label{conclrema}
In this paper, we explicitly computed the constant $C>0$ arising in the $L^2$-estimate \eqref{l2estimate} discussed in the Introduction, as stated in Theorem~\ref{maintheor}(iii). The constant $C>0$ in Theorem \ref{maintheor}(iii) satisfies the following bound:
\begin{equation} \label{constatcompar}
C \leq \min \left( \frac{4K_1 (d-1)^2 |U|^{\frac{2}{d}} }{d^2 \gamma \lambda},\; \frac{\sqrt{K_1}}{\alpha} \right),
\end{equation}
since
$$
K_1\left( \frac{d^2 \gamma \lambda}{4(d-1)^2 |U|^{\frac2d}} + \alpha \right)^{-1} \leq  \frac{4K_1 (d-1)^2 |U|^{\frac{2}{d}} }{d^2 \gamma \lambda}
$$
and
$$
K_1^{\frac12}\left( \frac{d^2 \gamma \lambda}{4K_1(d-1)^2 |U|^{\frac{2}{d}}}  +\alpha  \right)^{-1}	\leq \frac{\sqrt{K_1}}{\alpha}.
$$
Thus, $C \leq \frac{\sqrt{K_1}}{\alpha} \leq \frac{K_1}{\alpha}$, which shows that the $L^2$-estimate in Theorem~\ref{maintheor}(iii) is strictly better than \eqref{l2contraesticl}.
We emphasize that the constant $K_1 \geq 1$ originates from Theorem~\ref{existrho}(ii) and is independent of both $\gamma$ and $\alpha$. Consequently, the constant $C$ in Theorem~\ref{maintheor}(iii) becomes smaller as $\alpha$ or $\gamma$ increases.
For a more precise quantitative analysis on the constant $C$, it will be important to compute $K_1 \geq 1$ explicitly. In fact, $K_1$ is a constant derived from the Harnack inequality, and as is well known, a constant appearing in the Harnack inequality obtained through general theoretical proofs typically grows quite large due to the complexity of the underlying arguments.
Developing mathematical techniques for obtaining sharp bounds on $K_1 \geq 1$ will provide significant insight into the quantitative analysis of the constant $C>0$ established in this paper.

\centerline{}
\centerline{}
Haesung Lee\\
Department of Mathematics and Big Data Science,  \\
Kumoh National Institute of Technology, \\
Gumi, Gyeongsangbuk-do 39177, Republic of Korea, \\
E-mail: fthslt@kumoh.ac.kr, \; fthslt14@gmail.com
\end{document}